\documentclass[a4paper]{article}

\usepackage[english]{babel}
\usepackage[utf8x]{inputenc}
\usepackage{amsmath}
\usepackage[]{amsfonts}
\usepackage{amsthm}
\usepackage{graphicx}
\usepackage[colorinlistoftodos]{todonotes}
\usepackage{hyperref}

\newtheorem{definition}{Definition}
\newtheorem{proposition}{Proposition}
\newtheorem{lemma}{Lemma}

\title{Additive Collatz Trajectories}
\author{Aalok Thakkar, Mrunmay Jagadale}
\date{August 8, 2016}

\begin{document}
\maketitle

\begin{abstract}
Collatz Conjecture (also known as Ulam's conjecture and 3x+1 problem) concerns the behavior of the iterates of a particular function on natural numbers. A number of generalizations of the conjecture have been subjected to extensive study. This paper explores Additive Collatz Trajectories, a particular case of a generalization of Collatz conjecture and puts forward a sufficient and necessary condition for looping of Additive Collatz Trajectories, along with two minor results. An algorithm to compute the number of equivalence classes when natural numbers are quotiented by the limiting behavior of their corresponding trajectories is also proposed.

\end{abstract}

\section{Introduction}

\subsection{The Collatz Conjecture}

The $3x + 1$ problem is most simply stated in terms of the Collatz function $C(x)$ defined on integers as:

\[ C(x) = \begin{cases} 
      3x + 1 & \text{if } x \equiv 1 \mod 2\\
      \frac{x}{2} &\text{if } x \equiv 0 \mod 2\\ 
   \end{cases}
\]

The Collatz Conjecture states that every for every $m \in \mathbb{N}$, there exists $k \in \mathbb{N}$ such that iterate $T^{(k)}(m) = 1$. \cite{lagarias}. One natural generalization of Collatz function would be to consider an arbitrary affine linear map of $x$ instead of $3x + 1$. 

\subsection{Generalization}

We introduce the concept of a generalized Collatz function $C_{a,d,m}'(x)$ as:

\[ C'_{a,d,m}(x) = \begin{cases} 
      mx + a & \text{if } x \not\equiv 0 \mod d\\
      \frac{x}{d} &\text{if } x \equiv 0 \mod d\\
   \end{cases}
\]

where $x,a,d,m \in \mathbb{N}$. For an arbitrary choice of $(x, a, d, m)$ and for sufficient large values of $k \in \mathbb{N}$, we would like to explore the nature of $C_{a,d,m}'^{(k)}(x)$ \cite{generalized}. In this paper, we look at a particular case of the generalized Collatz function, which we term as the Additive Collatz function. An additive Collatz function $T_{a,d}(x)$ is defined as $C_{a,d,1}'(x)$. 

\subsection{Terminology}

We use the following definitions: 

\begin{definition}
An additive Collatz trajectory $O_{a,d}$ of an integer $x$ is the infinite tuple

$$O_{a,d}(x) = (x, T_{a,d}(x), T_{a,d}^{(2)}(x), ... )$$
\end{definition}

\begin{definition}

A trajectory $O = (o_0, o_1, o_2 ... )$ is said to loop if $\exists k, N \in \mathbb{N}$ such that $\forall n > N$, $o_n = o_{n+k}$.

\end{definition}

\begin{definition}

Given $a, d \in \mathbb{N}$, two natural numbers $x_1$ and $x_2$ are said to be equivalent under the orbit equivalence relation if $\exists n_1, n_2, N \in \mathbb{N}$ such that $\forall k > N$:

$$ T_{a,d}^{(k+n_1)}(x_1) = T_{a,d}^{(k+n_2)}(x_2) $$

\end{definition}

\begin{definition}

Given $a, d \in \mathbb{N}$, an orbit is an element of partition of  $\mathbb{N}$ under the orbit equivalence relation. 

\end{definition}

\section{Analysis of Additive Collatz Trajectories}

The limiting behavior of an additive Collatz trajectory is identified based on the eventual formation of loops. It is a straightforward observation that if $a$ and $d$ are not co-prime, then the trajectory does not necessarily loop.

\begin{proposition}

For non-coprime $a,d \in \mathbb{N}$, there exists $x \in \mathbb{N}$ such that $O_{a,d}(x)$ does not loop.

\end{proposition}

\begin{proof} Consider $ r \not\equiv 0 \mod \gcd(a,d) $. Then, $O_{a,d}(r)$ does not loop. This can be proved by induction on natural numbers. \\

Claim: $\forall k \in \mathbb{N}, T_{a,d}^{(k)}(r) = r + ak$ \\

Base Case: Trivially true for $k = 0$\\

Induction Hypothesis: $\exists k \in \mathbb{N} \ni T_{a,d}^{(k)}(r) = r + ak$ \\

Induction Step: Let $\delta= \gcd(a,d)$. 

$$T_{a,d}^{(k)}(r) = r + ak \equiv r \mod \delta \implies T_{a,d}^{(k)}(r) \not\equiv 0 \mod d$$
$$\implies T_{a,d}^{(k+1)}(r) = T_{a,d}^{(k)}(r) + a = r + ka + a = r + (k+1)a$$

Hence, the trajectory is an increasing progression, and does not loop. \\  
\end{proof}

It can be observed that if $r$ is not a multiple of $\gcd(a,d)$, then $O_{a,d}(r)$ does not loop. We would now like to analyze the converse of this statement.

\begin{lemma}
Given $a,d \in \mathbb{N}$, if $r \equiv 0 \mod \delta$, $$T_{a,d}^{(k)}(r) = \delta T_{\frac{a}{\delta},\frac{d}{\delta}}^{(k)}\Big( \frac{r}{\delta} \Big)$$ where $\delta = \gcd(a,d)$
\end{lemma}

If $r$ is a multiple of $\gcd(a,d)$, then Lemma 1 permits us to reduce our analysis to a case where $a$ and $d$ are co-prime. We now only consider the cases where $a$ and $d$ are co-prime and prove that for all natural numbers $x$, the additive Collatz trajectory $O_{a,d}(x)$ loops. This analysis is divided into three propositions.

\begin{proposition}
For co-prime $a,d \in \mathbb{N}$ given $x \in \mathbb{N}$, there exists $N(x) \in \mathbb{N}$ such that $T_{a,d}^{N(x)}(x) \leq a$ 
\end{proposition}

\begin{proof}
Consider $(n_i)$ as a sub-trajectory of $O_{a,d}(x)$ such that $$n_0 = x$$ $$n_i = T_{a,d}^{(z_i)}(x)$$

where $T_{a,d}^{(z_i - 1)}(x) = dT_{a,d}^{(z_i)}(x)$. As $a$ and $d$ are co-prime, B\'{e}zout's Lemma forces the existence of such a sub-trajectory. Let $y_i = z_{i+1} - z_i - 1$. Then,

\begin{equation} \label{1}
n_{i+1} = \frac{n_i + ay_i}{d}
\end{equation}

On solving the recursion, 

$$n_{k+1} = \frac{n_0 + a\sum_{i = 0}^k y_id^i}{d^{k+1}}$$

By definition, $y_i$ is the least non-negative integer such that $n_i + ay_i$ is divisible by $d$, and hence $y_i$ is strictly less than $d$. Hence,

$$n_{k+1} \leq \frac{n_0 + a\sum_{i = 0}^k (d-1)d^i}{d^{k+1}}$$

Which simplifies to

$$n_{k+1} \leq \frac{n_0 - a}{d^{k+1}} + a$$

For sufficiently large $k$, $d^{k+1} > n_0 - a$. Hence, there is an element in the trajectory which is less than or equal to $a$. \\ 
\end{proof}

\begin{proposition}
For co-prime $a,d \in \mathbb{N}$, $O_{a,d}(x)$ loops for all $x \leq a$.
\end{proposition}

\begin{proof}
Equation (1) and the fact that $y_i  \leq (d-1) $ implies that if $ n_i \leq a$,  then $ n_{i+1} \leq \dfrac{a + a(d-1)}{d} = a $ . Thus if  $x = n_0   \leq a $ then $ \forall j \in \mathbb{N} $ $n_j \leq a $. As there are only finitely many natural numbers less than or equal to $a$, trajectory loops.
\end{proof}

\begin{proposition}
Given $a,d \in \mathbb{N}$, $O_{a,d}(r)$ loops for all $r \equiv 0 \mod \gcd(a,d)$.
\end{proposition}

\begin{proof}
By Lemma 1, $O_{a,d}(r)$ loops if and only if $O_{\frac{a}{\delta},\frac{d}{\delta}}(\frac{r}{\delta})$ loops where $\delta = \gcd(a,d)$.

Let $a' = a/\delta$, $d' = d/\delta$ and $r' = r/\delta$. By Proposition 2, there exists $N(r')$ such that $T_{a,d}^{(N(r'))}(r') \leq a$. Let $T_{a,d}^{(N(r'))}(r') = k$. As the each term of the trajectory depends only on the previous term, $O_{a',d'}(r')$ is equal to $O_{a',d'}(k)$.
By Proposition 3, $O_{a',d'}(k)$ loops, hence $O_{a',d'}(r')$ loops. This implies that $O_{a,d}(r)$ loops. \\
\end{proof}

A straightforward implication of Proposition 4 is that if $a$ and $d$ are co-prime, then $\forall x \in \mathbb{N}$, $O_{a,d}(x)$ loops.

\section{Orbit Counting}

For co-prime $a$ and $d$, we have shown that $O_{a,d}(x)$ loops. We would now like to characterize the number of unique loops possible. For this, the concept of orbit equivalence relation is proposed. Two natural numbers are said to be equivalent, if their trajectories are the same, eventually. Formally speaking, given $a, d \in \mathbb{N}$, two natural numbers $x_1$ and $x_2$ are said to be equivalent under the orbit equivalence relation if $\exists n_1, n_2, N \in \mathbb{N}$ such that $\forall k > N$:

$$ T_{a,d}^{(k+n_1)}(x_1) = T_{a,d}^{(k+n_2)}(x_2) $$

Under the orbit equivalence relation, the set of natural numbers can be partitioned into equivalence classes. The number of equivalence classes is the same as the number of unique loops of trajectories possible. 

Proposition 3 and Proposition 4 imply that the number of unique loops formed by $O_{a,d}(x)$ where $x \in \mathbb{N}$ is the same as the number of unique loops formed by $O_{a,d}(x)$ where $x \in \mathbb{Z}/a\mathbb{Z}$. In order to count the equivalence classes, we identify them with orbits of a group action and use P\'{o}lya's Enumeration Theorem to count them. 
\subsection{Group Action}

Consider the sub-trajectory $(n_i)$ as defined in the proof of Proposition 2. We observe that given a trajectory $O_{a,d}(x)$, one can construct the sub-trajectory $(n_i)$ and vice-versa. Hence, the number loops of the sub-trajectories have a one-to-one correspondence with the loops of the trajectories.

By definition of the sub-trajectory $(n_i)$

$$n_{i+1} = d^{-1}n_i \mod a$$

For some $k$, all elements of the sub-trajectory are eventually smaller than $a$. This allows us to identify the limiting behavior of the sub-trajectory with $$(n_k,\: d^{-1}n_k \mod a,\: d^{-2}n_k \mod a,\: d^{-3}n_k \mod a,\: ...)$$ where $n_k \leq a$. Each element of the sub-trajectory is a power of $d^{-1}$ multiplied by $n_k$ and hence, the sub-trajectory can be identified with the action of the group of negative powers of $d$ on $n_k$ under multiplication modulo $a$.

Under the binary operation of multiplication modulo a natural number $a$, the numbers co-prime to $a$ (modulo $a$) form an Abelian group denoted by $(\mathbb{Z}/a\mathbb{Z})^*$. For some $d$ in $(\mathbb{Z}/a\mathbb{Z})^*$, let $H$ be the subgroup generated by $d$.

$$H = \{ d^i : i \in \mathbb{Z} \}$$

Then, the number of orbits formed by quotienting natural numbers by the nature of limiting behavior of the additive Collatz trajectory $O_{a,d}(x)$ is given by the number of orbits under the action of $H$ on $\mathbb{Z}/a\mathbb{Z}$ under the binary operation multiplication modulo $a$. 

\subsection{Computation}

Let $\xi(a,d)$ denote the number of orbits of $(\mathbb{Z}/n\mathbb{Z})$ when $H$ acts on it. By P\'{o}lya's Enumeration Theorem, we have:

$$\xi(a,d) = \frac{1}{|H|} \sum_{x \in S} |H_x|$$

where $H_x = \{g \in H: gx \equiv x \mod a \}$. $|H_x|$ can be computed by finding the number of solutions for $d^t \in H$ in the equation:

\begin{equation}
d^tx \equiv x \mod a
\end{equation}

Let $m_x = \gcd(x,a)$, $p_x = \frac{a}{m_x}$ and $q_x = \frac{x}{m_x}$. On substitution in equation (2), we have:

\begin{equation}
d^tm_xq_x \equiv m_xq_x \mod (m_xp_x)
\end{equation}

As $\gcd(p_x,q_x) = 1$, the number of solutions to equation (3) is same as the number of solutions to :

$$d^t \equiv 1 \mod p_x$$

Let the smallest solution to equation be termed as $\alpha_{p_x} (d)$. Therefore the total number of solutions to equation (2) are:
$\frac{|H|}{\alpha_{p_x}(d)}$

Hence, we have:

$$\xi(a,d) = \frac{1}{|H|}\sum_{x \in S}\frac{|H|}{\alpha_{p_x}(d)} = \sum_{x \in S} \frac{1}{\alpha_{p_x}(d)}$$

For each $p_x$, $q_x$ takes values co prime to $p_x$, hence, by counting repetitions, we get:

$$\xi(a,d) = \sum_{f | a} \frac{\phi(f)}{\alpha_{f}(d)}$$

where $\phi$ is the Euler-totient function. 

\subsection{Upper and Lower Bounds}

We can set a lower bound on $\xi(a,d)$ by considering the Carmichael function $\lambda$. 

$$\lambda(m) = \max \{ \alpha_m(d) : d \in (Z/mZ) \}$$

Hence, 

$$\xi(a,d) = \sum_{f | a} \frac{\phi(f)}{\alpha_{f}(d)} \geq \sum_{f | a} \frac{\phi(f)}{\lambda(f)} = \xi_{inf}(a)$$

One can further claim that $\xi_{inf}(a)$ is a strong lower bound for $\xi(a,d)$ as for every $a$, there exists $d$ such that for all factors $f$ of $a$, $\alpha_{f}(d) = \lambda(f)$. The proof of this claim relies on the decomposition of $(\mathbb{Z}/n\mathbb{Z})^*$ into cyclic groups.

The strong upper bound for $\xi(a,d)$ is $a$, which is attained when $d$ is $1$.

\subsection{Applications}

The computation of $\xi(a,d)$ employs factorization as well as the discrete logarithm function (in computation of $\alpha_{f}(d)$). There are no known efficient algorithms to compute either of them, making computation of $\xi(a,d)$ difficult. This difficulty can be employed for public-key cryptography. Knowing the prime factorization of $a$, can make the computation of $\xi(a,d)$ easier. Consider two primes $p$ and $q$. Then, for some $d$ co-prime to $pq$, 

$$\xi(pq,d) = 1 + \frac{\phi(p)}{\alpha_{p}(d)} + \frac{\phi(q)}{\alpha_{q}(d)} + \frac{\phi(pq)}{\alpha_{pq}(d)}$$

On simplification, we have:

$$\xi(pq,d) = 1 + \frac{p-1}{\alpha_{p}(d)} + \frac{q-1}{\alpha_{q}(d)} + \frac{(p-1)(q-1)}{\alpha_{p}(d)\alpha_{q}(d)}\gcd(\alpha_{p}(d),\alpha_{q}(d))$$

Computing $\xi(pq,d)$ would be cumbersome without using equation (9), however, it is much simpler using the prime factorization. This provides much hope for the possibility of design of a public key cryptography algorithm or key exchange system using additive Collatz trajectories.

\section{Results}

This paper puts forward the concept of Additive Collatz Trajectories and provides an analysis of their limiting behavior. A necessary and sufficient condition is provided for eventual looping of the Additive Collatz trajectories, along with a formula to compute the number of unique trajectories possible up to the orbit equivalence relation.

\section{Further Scope}

\subsection{Generalized Collatz Trajectories}

The spirit and strategy of this paper can be used to deal with generalized Collatz trajectories. One immediate result is that if the equation:

\begin{equation}
m^{r+1} x + a ( m^r + m^{r-1} + ... +m +1)  \equiv 0 \mod d
\end{equation}
does not have a solution for any $r$, then for all $x \in \mathbb{N}$, the trajectory formed by iteration of $C_{a,d,m}'(x)$ would not loop. Also, if $ m \equiv 1 \mod d $, then the equation will always have a solution. We can then define a sub-trajectory similar to that done in Proposition 2 as:
$$n_0 = x$$
$$n_{i} = C_{a,d,m}'^{(z_i)}(x)$$
where $C_{a,d,m}'^{(z_i - 1)}(x) = dC_{a,d,m}'^{(z_i)}(x)$.
We will be able to show that:

If $ d \not | n_i $ $$ n_{i+1} = \dfrac{m^{r_i}n_i + a( m^{r_i-1} + ... +m +1 )}{d}$$ where $ r  \equiv  -a^{-1}n_i \mod d$ and $ 0 < r_i \leq (d-1) $ \\
 
else,$$  n_{i+1} = \dfrac{n_i}{d} $$
\subsection{Public-key Cryptography}

As mentioned in subsection 3.4, there is a hope for developing a public-key cryptography system that relies on Additive Collatz Trajectories, particularly on counting the number of equivalence classes formed under the orbit equivalence relation. Much effort and study is required to design an implementable cryptography design, as there are a number of challenges. Firstly, there is no trivial characteristic that is common among the elements of an orbit equivalence class. Secondly, the formula for counting the number of orbit equivalence classes employs a number of functions, hence there is no natural way to compute the inverse for decryption. Lastly, the formula uses the discrete logarithm function which cannot practically be computed for larger cases. One must deal with these challenges in the process of an encryption algorithm design using the results proved in this paper.

\end{document}